\documentclass[a4paper]{amsart}
\usepackage[all]{xy}

\usepackage{latexsym}
\usepackage{amsmath}
\usepackage{amssymb}

\newtheorem{theorem}{Theorem}

\newtheorem{proposition}{Proposition}
\newtheorem{corollary}{Corollary}

\theoremstyle{definition}
\newtheorem{definition}{Definition}
\newtheorem{remark}{Remark}
\newtheorem{example}{Example}

\newcommand{\p}{{\mathbb P}}

\newcommand{\g}{{\mathbb G}}

\newcommand{\an}{{\mathcal O}}

\begin{document}

\title[Surfaces containing a family of plane curves]{Surfaces containing a family of plane curves not forming a fibration}
\author{Jos\'e Carlos Sierra}
\address{Instituto de Ciencias Matem\'aticas (ICMAT), Consejo Superior de Investigaciones Cient\'{\i}ficas (CSIC), Campus de
Cantoblanco, 28049 Madrid, Spain} \email{jcsierra@icmat.es}
\thanks{Research supported by the ``Ram\'on y Cajal" contract RYC-2009-04999 and the project MTM2009-06964 of MICINN}

\begin{abstract}
Extending some previous results of \cite{s-t} and \cite{b-s}, we
complete the classification of smooth surfaces swept out by a
$1$-dimensional family of plane curves that do not form a fibration.
As a consequence, we characterize manifolds swept out by a
$1$-dimensional family of hypersurfaces that do not form a
fibration.
\end{abstract}
\maketitle

\section*{Introduction}

In a classical paper \cite{seg2}, Segre characterized the Veronese
surface, the rational cubic scroll and the cones as the only
surfaces in $\p^N$ containing a $2$-dimensional family of plane
curves when $N\geq 4$. In the late 80's, the influential article
\cite{e-p} on surfaces in $\p^4$ led to the study of surfaces swept
out by a $1$-dimensional family of plane curves. Surfaces in $\p^4$
ruled in conics were classified in \cite{e-s} (see also \cite{b-r})
and \cite{a-d-s}, and surfaces in $\p^4$ fibred by plane curves of
arbitrary degree were classified in \cite{rane}. In that paper it
was suggested to study surfaces in $\p^4$ containing a family of
plane curves not forming a fibration (see \cite[Remark 0.3]{rane}).
The classification of such surfaces has recently been obtained in
\cite{s-t} and \cite{b-s}.

To complete the picture, in this note we focus on surfaces in $\p^N$
containing a $1$-dimensional family of plane curves when $N\geq 5$.
Fibrations can be obtained by taking a divisor on a scroll of planes
over a curve, so these surfaces do not exhibit any special property.
However, the situation is different when the surface is swept out by
a family of plane curves not forming a fibration. This is due to the
fact that families of planes in $\p^N$ such that any two of them
intersect are special if $N\geq 5$, as Morin already noticed in
\cite{morin} (cf. Theorem \ref{thm:morin}). Our results, combined
with those of \cite{s-t} and \cite{b-s} for $N=4$, can be summarized
in the following way:

\begin{theorem}\label{thm:main}
Let $X\subset\p^N$ be a non-degenerate smooth linearly normal
surface containing a $1$-dimensional family of plane curves not
forming a fibration. If $N\geq 4$ then one of the following holds:

\begin{enumerate}
\item[(i)] $X\subset\p^N$ is contained in a rational normal scroll
$S:=S_{0,a,b}\subset\p^{N=a+b+2}$, with $a\in\{0,1\}$, and it is
linked to $N-3$ rulings of $S$ by the complete intersection of $S$
and a hypersurface in $\p^N$;

\item[(ii)] $X\subset\p^5$ is the second symmetric product
of a smooth plane curve, and the embedding is given by the embedding
of the second symmetric product of the plane as the secant variety
of the Veronese surface;

\item[(iii)] $X$ is a ruled surface with invariant $e=-1$ over an elliptic curve,
and it is embedded either in $\p^4$ by $|C_0+\mathcal Lf|$, where
$\deg(\mathcal L)=2$, or in $\p^5$ by $|2C_0+\mathcal Lf|$, where
$\deg(\mathcal L)=1$.
\end{enumerate}
\end{theorem}

As a consequence of Theorem \ref{thm:main}, we obtain the following:

\begin{corollary}\label{cor:main}
Let $Y\subset\p^N$ be a non-degenerate manifold of dimension $n\geq
3$ containing a $1$-dimensional family of hypersurfaces not forming
a fibration. If $N-n\geq 2$ then $n=3$, $N=5$ and $Y\subset\p^5$ is
linked to a linear $\p^3$ by the complete intersection of a quadric
cone of rank $4$ and a hypersurface in $\p^5$.
\end{corollary}

Corollary \ref{cor:main} can be seen as the counterpart to the
problem of classifying manifolds of dimension $n$ in $\p^N$ fibred
by hypersurfaces (or even linear spaces) of dimension $n-1$.
According to the Barth-Larsen theorem, the Picard group of a
manifold is generated by the hyperplane section if $N\leq 2n-2$. On
the other hand, hypersurface fibrations are not special when $N\geq
2n+1$ so this problem makes sense only if $2n-1\leq N\leq 2n$.
Manifolds in $\p^{2n-1}$ fibred by linear spaces of dimension $n-1$,
usually called \emph{scrolls}, were classified in \cite{kl}, and
manifolds in $\p^{2n-1}$ swept out by hypersurfaces (either forming
a fibration or not) have recently been classified in \cite{sierra}.
However, the situation seems to be more complicated for $N=2n$,
where the irregularity of the manifold needs not to be zero. In
fact, the classification of scrolls in $\p^{2n}$ has only been
obtained for $n\leq 4$ (see \cite{i-toma}), and, to the best of the
author's knowledge, hypersurface fibrations in $\p^{2n}$ remain
unexplored for $n\geq 3$.

The paper is organized as follows. In the preliminary section we
study surfaces embedded in singular rational normal scrolls of
planes, and we also recall some facts on families of intersecting
planes that will be crucial in the sequel. In Section
\ref{section:proof} we prove Theorem \ref{thm:main} and Corollary
\ref{cor:main}. Finally, in Section \ref{section:tribute} we point
out some other consequences of Morin's work on families of
intersecting linear spaces.

\section{Preliminaries}

\subsection{Notation and conventions}

We work over the field of complex numbers. We adopt the following notation and conventions:

$\p^N$: projective space of dimension $N$

${\p^N}^*$: dual projective space

$\g(k,N)$: Grassmann variety of $k$-planes in $\p^N$

$a_0,\dots,a_n$: sequence of non-negative integers satisfying $0\leq
a_0\leq\dots\leq a_n$

$S:=S_{a_0,\dots,a_n}$ rational normal scroll of $n$-planes in
$\p^{d_S+n}$, where $d_S:=\sum_{i=0}^n a_i$

$\tilde S:$ projectivization of the vector bundle
$E:=\an_{\p^1}(a_0)\oplus\dots\oplus\an_{\p^1}(a_n)$ over $\p^1$

$X$: integral subvariety of $\p^N$

$\tilde X$: strict transform in $\tilde S$ of $X\subset S$

$T_xX$: embedded tangent space to $X\subset\p^N$ at $x$

$\sim$: rational equivalence of cycles

$\{C_b\}_{b\in B}$: $1$-dimensional family of plane curves on $X$
not forming a fibration

$\Sigma$: $1$-dimensional family of planes in $\g(2,N)$
corresponding to $\{C_b\}_{b\in B}$

$V_{\Sigma}$: union of the planes of $\Sigma$ in $\p^N$

%$Y$: non degenerate manifold in $\p^N$ of dimension $n\geq 3$

%$\{D_b\}_{b\in B}$: $1$-dimensional family of hypersurfaces on $Y$ not forming a fibration

$S^2Z$: second symmetric product of a variety $Z$

$x+y$: point of $S^2Z$ corresponding to the unordered pair of points
$x,y\in Z$

$v_2(\p^k)$: $2$-Veronese manifold in $\p^{k(k+3)/2}$ given by
$|\an_{\p^k}(2)|$

%$\phi$: natural map from $S^2\p^2$ to $\p^5$, whose image is the
%secant variety of $v_2(\p^2)$

$\phi_k$: map from $S^2\p^k$ to $\p^{k(k+3)/2}$, whose image is the
secant variety of $v_2(\p^k)$

$v_k(\p^2)$: $k$-Veronese surface in $\p^{k(k+3)/2}$ given by
$|\an_{\p^2}(k)|$

\subsection{Singular rational normal scrolls and Roth varieties}

Let $S:=S_{a_0,\dots, a_n}\subset\p^{N=d_S+n}$ be a rational normal
scroll of $n$-planes of degree $d_S:=\sum_{i=0}^n a_i=N-n$. In this
paper we are interested in two particular cases, namely $0=a_0<a_1$
and $0=a_0=a_1<a_2$. In both cases $S\subset\p^N$ is an
$(n+1)$-dimensional cone of vertex $p$ (a point) and $L$ (a line),
respectively. We set some notation that will be
used throughout the paper. Let $\tilde S:=\p(E)$ be the
projectivization of the vector bundle
$E:=\an_{\p^1}(a_0)\oplus\dots\oplus\an_{\p^1}(a_n)$ over $\p^1$. So
$\tilde S$ is a desingularization of $S$ with natural projections
\[
\xymatrix{\tilde S \ar[d]_{\pi_1} \ar[r]^<<<<<{\pi_2} & S\subset\p^N\\
\p^1}
\]
%$\pi_1:\tilde S\to\p^1$ gives the $\p^n$-bundle structure of $\tilde
%S$, and $\pi_2(\tilde S)=S$.
Let $F$ denote a fibre of $\pi_1$, and let $H$ denote the pullback
by $\pi_2$ of a hyperplane section of $S\subset\p^N$. These two
divisors generate the Chow ring $A(\tilde S)$ of $\tilde S$, with
relations $F^2\sim 0$ and $H^{n+1}\sim d_SH^n\cdot F$. For any
effective Weil divisor $X$ on $S$, let $\tilde X$ denote the strict
transform of $X$ on $\tilde S$. So $\tilde X \sim \alpha H+\beta F$
for some integers $\alpha,\beta$.

Manifolds $X$ of dimension $n\geq 2$ contained in
$S_{0,0,a_2,\dots,a_n}$ such that $L\subset X$ were widely studied
in \cite{ilic}, extending to higher dimensions the study of surfaces
contained in rational normal scrolls of planes initiated in Roth's
classical paper \cite[\S3.5]{roth}, and they were called \emph{Roth
varieties} (see \cite[Definition 3.1]{ilic}). Most of the results
obtained in Ilic's paper also hold for manifolds $X$ of dimension
$n\geq 2$ contained in $S_{0,a_1,a_2,\dots,a_n}$ such that $p\in X$
(see \cite[Remark 5.7]{ilic}). However, there is one case that does
not seem to be considered there. More precisely, for $n=2$ it might
happen that $\pi_2|_{\tilde X}:\tilde X\to X$ is the blowing-up of
$X$ at $p$ instead of an isomorphism (as announced in \cite[Remark
5.7]{ilic} concerning Theorem 3.7). So we state \cite[Remark
5.7]{ilic} in the following way:

\begin{theorem}\label{thm:ilic}
Let $X\subset\p^N$ be a manifold of dimension $n$ contained in a
rational normal scroll $S:=S_{0,a_1,\dots,a_n}$ with $a_i\geq 1$ and
$\sum_{i=0}^n a_i=N-n$ such that $p\in X$. Then either
$\pi_2|_{\tilde X}:\tilde X\to X$ is an isomorphism and $\tilde
X\in|\alpha H+F|$, or $n=2$, $\pi_2|_{\tilde X}:\tilde X\to X$ is
the blowing-up of $X$ at $p$, $S=S_{0,1,b}\subset\p^{N=b+3}$,
$\tilde X\in|\alpha H-bF|$ and $X\subset\p^N$ is linked to $N-3$
rulings of $S$ by the complete intersection of $S$ and a
hypersurface of degree $\alpha$ in $\p^N$. Furthermore, in both
cases, such a manifold exists for every integer $\alpha\geq 1$.
\end{theorem}

\begin{proof}
Let $\xi:=\pi_2^{-1}(p)\subset\tilde S$. If $\pi_2|_{\tilde X}:\tilde X\to X$ is an isomorphism then $\tilde X\cdot\xi=1$, so we get $1=\tilde X\cdot\xi=(\alpha H+\beta F)\cdot\xi=\beta$ and therefore
$\tilde X\in|\alpha H+F|$, as claimed in \cite[Remark 5.7]{ilic}. Otherwise,
$\pi_2|_{\tilde X}:\tilde X\to X$ contracts the curve
$\xi$ by Zariski's main theorem. Hence
we deduce that $n=2$ (if $n\geq 3$ then $\pi_2|_{\tilde X}$ would be
a small contraction and $X$ would be singular at $p$). This happens
if and only if $\xi\cong\p^1$ is a $(-1)$-curve on $\tilde X$, i.e.
if and only if $\xi^2=-1$. The exact sequence of normal bundles
\[
0\to N_{\xi/\tilde X}\to N_{\xi/\tilde S}\to (N_{\tilde X/\tilde
S})|_{\xi}\to 0
\]
turns out to be
\[
0\to\an_{\p^1}(-1)\to\an_{\p^1}(-a)\oplus\an_{\p^1}(-b)\to\an_{\p^1}(\beta)\to
0,
\]
as $\xi\sim H^2-(a+b)H\cdot F$ and $(N_{\tilde X/\tilde
S})|_{\xi}=\tilde X\cdot\xi=(\alpha H+\beta F)\cdot(H^2-(a+b)H\cdot
F)=\beta$. Therefore, $\beta=1-a-b$ and
$\an_{\p^1}(-a)\oplus\an_{\p^1}(-b)\in\text{Ext}^1(\an_{\p^1}(\beta),\an_{\p^1}(-1))=H^1(\p^1,\an_{\p^1}(-1-\beta))=0$,
so $a=1$ and $b=-\beta$, whence $\tilde X\in|\alpha H-bF|$ and $X$
is linked to $b=N-3$ rulings of $S$ by the complete intersection of
$S$ and a hypersurface of degree $\alpha$ in $\p^N$.

Let us show now the existence of $X$ for every integer $\alpha\geq
1$. Since $H+F$ is very ample on $\tilde S$ we deduce that $\alpha
H+F$ is also very ample for any $\alpha\geq 1$, whence a general
$\tilde X\in|\alpha H+F|$ is smooth by Bertini's theorem (and hence
irreducible). Assume now $n=2$, $a=1$ and $\tilde X\in|\alpha
H-bF|$. Since $\xi\subset\tilde X$, we cannot directly deduce from
Bertini's theorem that a general $\tilde X\in|\alpha H-bF|$ is
smooth. However, we can argue as in \cite[pp.~60--61]{harris}. Let
$\p^2\subset\p^N$ be the plane corresponding to the embedding
$S_{0,1}\subset S_{0,1,b}$ (which is unique if $b>1$). Hence
$\tilde\p^2\in|H-bF|$ and
$\tilde\p^2\cong\p(\an_{\p^1}\oplus\an_{\p^1}(-1))$, so in
particular the statement for $\alpha=1$ is proved. Let
$\sigma:Bl_{\xi}(\tilde S)\to\tilde S$ be the blowing-up of $\tilde
S$ along $\xi$, and let $Bl(\tilde X)$ denote the strict transform
of $\tilde X$ on $Bl_{\xi}(\tilde S)$. Note that
$\sigma|_{\sigma^{-1}(\xi)}:\sigma^{-1}(\xi)\to\xi$ is a
$\p^1$-bundle over $\xi$, whose fibre is denoted by $f$, and that
$Bl(\tilde\p^2)\cdot f=1$. If $\alpha\geq 2$, we remark that
$\dim(|\alpha H-bF|)\geq 1$ and that $Bl(\tilde X)$ cuts out the
complete linear system $\an_{f}(1)$, and so it has no base points on
$\sigma^{-1}(\xi)$. It follows that $|Bl(\tilde X)|$ is a
base-point-free linear system of positive dimension, whence its
general element $V$ is smooth by Bertini's theorem. If $V$ would be
reducible, say $V=V'+V''$, then $V'\cap V''\neq\emptyset$ and $V$
would be singular. Therefore $V$ is irreducible. Furthermore,
$$Bl(\tilde X)\cdot f = (Bl(\tilde\p^2)+(\alpha-1)Bl(H))\cdot
f=1,$$ as $\tilde X\sim\tilde\p^2+(\alpha-1)H$ and $H\cdot\xi=0$. So
$V\cdot f=1$ and therefore $\sigma|_V:V\to\sigma|_V(V)$ is an
isomorphism. In particular $\sigma|_V(V)\in|\alpha H-bF|$, and hence
$X:=\pi_2(\sigma|_V(V))$ is also smooth and irreducible.
\end{proof}

\begin{remark}\label{rem:tangente}
In Theorem \ref{thm:ilic}, there is a more geometric way to see that
$a=1$ when $\pi_2|_{\tilde X}:\tilde X\to X$ is the blowing-up of
$X$ at $p$. Note that the embedded tangent lines at $p$ to the plane
curves on $X$ sweep out the embedded tangent plane $T_{p}X$, so we
deduce that $T_{p}X\subset S$ and that $T_{p}X$ is not a ruling of
$S$. Therefore $a=1$, since for $a\geq 2$ the only planes contained
in $S$ are the rulings.
\end{remark}

\begin{remark}
Surfaces $X\subset\p^N$ contained in singular rational normal
scrolls of planes whose vertex is a point $p$ were studied by Roth
in \cite[\S 3.5]{roth}. In \cite[p.~156]{roth}, it was stated that
surfaces containing $p$ are linked to $N-3$ planes by the complete
intersection of the scroll and a hypersurface in $\p^N$. We would
like to remark that this statement is correct when $\pi_2|_{\tilde
X}:\tilde X\to X$ is the blowing-up of $X$ at $p$, i.e. when $\tilde
X\in|\alpha H-bF|$, but it is no longer true when $\pi_2|_{\tilde
X}:\tilde X\to X$ is an isomorphism, i.e. when $\tilde X\in|\alpha
H+F|$.
\end{remark}

\begin{remark}
According to \cite[Definition 3.1]{ilic}, it is natural to extend
the definition of Roth varieties to smooth $n$-dimensional
subvarieties of $S_{0,\dots,0,a_{l+1},\dots,a_n}$ containing the
vertex of the scroll. For the sake of completeness, we point out
that $n$-dimensional subvarieties $X$ of
$S:=S_{0,\dots,0,a_{l+1},\dots,a_n}$ are always singular as soon as
$l\geq 2$. This follows from Zak's theorem on tangencies
\cite[Ch.~I, Corollary 1.8]{zak2}, arguing as in the proof of
Corollary \ref{cor:main}: if $X$ is smooth then the tangent space
$T_sS$ to $S$ at a general point $s\in S$ is tangent to $X$ along
$X_s:=\langle\p^l,s\rangle\cap X$, where $\p^l$ denotes the vertex
of $S$, and $\dim(X_s)=l$. Therefore,
$l=\dim(X_s)\leq\dim(T_sS)-n=n+1-n=1$. So the definition of (smooth)
Roth varieties only makes sense if $l\in\{0,1\}$.
\end{remark}

Let us look more closely at surfaces with a family of plane
curves not forming a fibration contained in rational normal scrolls:

\begin{example}
Let $S:=S_{0,0,b}\subset\p^{N=b+2}$, and let $L$ be its vertex. For
every integer $\alpha\geq 1$, let $\tilde X\in|\alpha H+F|$ be a
general divisor on $\tilde S$. Then $\tilde X$ is smooth,
$\pi_2|_{\tilde X}:\tilde X\to X$ is an isomorphism and $L\subset X$
(see \cite[Proposition 3.4]{ilic}). Furthermore, since $\tilde
X\in|\alpha H+F|\sim|(\alpha+1) H-(b-1)F|$ it follows that
$X\subset\p^N$ is linked to $b-1=N-3$ rulings of $S$ by the complete
intersection of $S$ and a hypersurface of degree $\alpha+1$ in
$\p^N$. We point out that $L$ is a base component of the pencil
$\{C_F\}_{F\in\p^1}:=\{L+\pi_2|_{\tilde X}(F|_{\tilde
X})\}_{F\in\p^1}$ of plane curves on $X$ of degree $\alpha+1$.
Moreover, we remark that $X$ is also fibred by the pencil
$\{C_F-L\}_{F\in\p^1}$ of plane curves of degree $\alpha$.
\end{example}

\begin{example}\label{ex:p}
Let $S:=S_{0,1,b}\subset\p^{N=b+3}$, and let $p$ be its vertex. For
every integer $\alpha\geq 1$, let $\tilde X\in|\alpha H-bF|$ be a
general divisor on $\tilde S$. Then $\tilde X$ is smooth,
$\pi_2|_{\tilde X}:\tilde X\to X$ is the blowing-up of $X$ at $p$,
and $X\subset S$ is a smooth surface linked to $b=N-3$ rulings of
$S$ by the complete intersection of $S$ and a hypersurface of degree
$\alpha$ in $\p^N$ by Theorem \ref{thm:ilic}. In this case, $p$ is
the base point of the pencil $\{C_F\}_{F\in\p^1}:=\{\pi_2|_{\tilde
X}(F|_{\tilde X})\}_{F\in\p^1}$ of plane curves on $X$ of degree
$\alpha$, and $(C_F)^2=(\pi^*_2(C_F))^2=(\xi+F|_{\tilde X})^2=1$.

%Let us look more closely at some particular cases:

%(i) if $b=1$ and $\alpha=2$, then $X=S_{1,2}\subset\p^4$ and
%$\{C_b\}_{b\in\p^1}$ is the pencil of conics passing through a
%point.

%(ii) if $b=2$ and $\alpha=2$, then $X=v_2(\p^2)\subset\p^5$ and
%$\{C_b\}_{b\in\p^1}$ is the pencil of conics passing through a
%point. %We point out that the Veronese surface in $S:=S_{0,1,2}$ is
%not linked to two rulings of $S$ by the complete intersection of $S$
%and a quadric $Q$ of $\p^5 $ since, otherwise, $Q$ intersects
%$T_pX=S_{0,1}\subset S$ (cf. Remark \ref{rem:tangente}) in a conic,
%which is not possible.

%(iii) if $b=2$ and $\alpha=3$, then $X$ is the blowing up of a Del
%Pezzo surface of degree $2$ at a point. More precisely,
%$\sigma:X\to\p^2$ is the blowing up of $\p^2$ at $8$ points
%$\{p_1,\dots,p_8\}$ embedded by
%$|6\sigma^*(L)-\sum_{i=1}^72E_i-E_8|$. Moreover,
%$\{C_b\}_{b\in\p^1}$ is a pencil of plane cubics
%$|3\sigma^*(L)-\sum_{i=1}^8E_i|$ corresponding to the strict
%transform of the pencil of plane cubics passing though
%$\{p_1,\dots,p_8\}$, and $p$ corresponds to the base point of this
%pencil.

%(iv) if $b=3$ and $\alpha=3$, then $X$ is a Del Pezzo surface of
%degree $1$. More precisely, $\sigma:X\to\p^2$ is the blowing up of
%$\p^2$ at $8$ points $\{p_1,\dots,p_8\}$ embedded by
%$|9\sigma^*(L)-\sum_{i=1}^8 3E_i|$. Moreover, $\{C_b\}_{b\in\p^1}$
%is a pencil of plane cubics $|3\sigma^*(L)-\sum_{i=1}^8E_i|$
%corresponding to the strict transform of the pencil of plane cubics
%passing though $\{p_1,\dots,p_8\}$, and $p$ corresponds to the base
%point of this pencil.
\end{example}

\subsection{Families of intersecting planes}\label{subsect:morin}

The main ingredient of the proof of Theorem \ref{thm:main} is a
classical result of Morin \cite{morin} on families of planes with
the property that any two of them intersect. We include it here for
the reader's convenience, as it might be difficult to find it in the
literature:

\begin{definition}
A subvariety $\Sigma\subset\g(2,N)$ such that any two planes of
$\Sigma$ intersect is said to be \emph{elementary} (see
\cite[p.~908]{morin}) if one of the following holds:
\begin{itemize}
\item $N\leq 4$;
\item there exists a point $p\in \p^N$ such that $p\in\Pi$ for every
$\Pi\in\Sigma$;
\item there exists some plane $\Lambda\subset\p^N$ such
that $\dim(\Lambda\cap\Pi)\geq 1$ for every $\Pi\in\Sigma$.
\end{itemize}
\end{definition}

\begin{theorem}[Morin \cite{morin}]\label{thm:morin}
Let $\Sigma\subset\g(2,N)$ be a non-degenerate (i.e. not contained
in any $\g(2,N-1)$) non-elementary integral subvariety of positive
dimension such that any two planes of $\Sigma$ intersect. Then $N=5$
and $\Sigma$ is contained in one of the following families:
\begin{enumerate}
\item[(i)] the $\infty^2$ planes of the conics of the Veronese
surface $v_2(\p^2)$ in $\p^5$;
\item[(ii)] the $\infty^2$ tangent planes to the Veronese
surface $v_2(\p^2)$ in $\p^5$;
\item[(iii)] one of the two $\infty^3$ planes contained in a smooth quadric $Q$ in
$\p^5$.
\end{enumerate}
\end{theorem}

\begin{remark}\label{rem:families}
(a) The family of tangent planes to the Veronese surface $v_2(\p^2)$
in $\p^5$ is isomorphic to $\p^2$, and the embedding
$\varphi_2:\p^2\hookrightarrow\g(2,5)$ is given by six (non-general)
sections of $\an_{\p^2}(1)^{\oplus 3}$. This follows from the
definition of the bundle of principal parts and Euler's exact
sequence.

(b) The family of planes of the conics of the Veronese surface
$v_2(\p^2)$ in $\p^5$ is isomorphic to $\p^2$, and the embedding in
$\g(2,5)$ corresponds to a vector bundle given by a resolution
$0\to\an_{\p^2}(-1)^{\oplus 3}\to\an_{\p^2}^6\to E\to 0$. In
particular, the family (i) corresponds to the family (ii) by the
identification $\g(2,\p^5)\cong\g(2,{\p^5}^*)$, and viceversa.

(c) The family of planes contained in a smooth quadric $Q$ of $\p^5$
is isomorphic to $\p^3$, and the embedding
$\varphi_3:\p^3\hookrightarrow\g(2,5)$ is given by the vector bundle
$\Omega_{\p^3}(2)$. In particular, the family (iii) is self-dual by
the identification $\g(2,\p^5)\cong\g(2,{\p^5}^*)$.
\end{remark}

Let us show some examples of surfaces swept out by plane curves not
forming a fibration coming from the families in Theorem
\ref{thm:morin}:

\begin{example}
Let $C\subset{\p^2}^*$ be an irreducible curve of degree $d$. Then
$v_2(\p^2)\subset\p^5$ is swept out by the family of conics
$\{v_2(L)\}_{L\in C}$, and a general point of $v_2(\p^2)$ is
contained in $d$ conics of the family (see \cite[Example 2.1]{s-t}).
The same property holds for the rational normal scroll of $\p^4$
(see \cite[Example 2.3]{s-t}).
\end{example}

\begin{example}\label{ex:phi}
Let $S^2\p^2$ denote the second symmetric product of $\p^2$, and let
the map $\phi_2:S^2\p^2\to Sec(v_2(\p^2))\subset\p^5$ denote the
embedding defined by
$$(X_0:X_1:X_2)+(Y_0:Y_1:Y_2)\mapsto(Z_{00}:Z_{01}:Z_{02}:Z_{11}:Z_{12}:Z_{22}),$$
where $Z_{ij}:=X_iY_j+X_jY_i$. Let us recall that the secant variety
and the tangent variety of $v_2(\p^2)$ in $\p^5$ coincide, as
$v_2(\p^2)\subset\p^5$ is $1$-defective (see \cite[Ch.~I, Theorem
1.4]{zak2}). Geometrically, $\phi_2$ is given by
$\phi_2(x+y)=T_{v_2(x)}v_2(\p^2)\cap T_{v_2(y)}v_2(\p^2)$ and
$\phi_2(x+x)=v_2(x)\in v_2(\p^2)$. In general, one can define in a
similar way an isomorphism
$\phi_k:S^2\p^k\to\text{Sec}(v_2(\p^k))\subset\p^{k(k+3)/2}$ with
the same geometric property (cf. Proposition \ref{prop:coarseee}).
For every smooth curve $C\subset\p^2$ of degree $d$ we obtain a
smooth surface $X:=\phi_2(S^2C)\subset\p^5$ containing a family of
plane curves $\{C_x\}_{x\in C}$ not forming a fibration, where
$C_x:=\phi_2(x+C)\cong C$ is a plane curve on $X$ of degree $d$ for
every $x\in C$ and $\langle C_x \rangle=T_{v_2(x)}v_2(\p^2)$.
Furthermore, we get $\an_X(1)\sim C_{x_1}+\dots+C_{x_d}$ for every
set of collinear points ${x_1},\dots,{x_d}\in C$ in $\p^2$. In fact,
if $L:=\langle{x_1},\dots,{x_d}\rangle$ then the hyperplane $H$ in
$\p^5$ corresponding to $2L\in H^0(\p^2,\an_{\p^2}(2))$ is tangent
to $v_2(\p^2)$ along the conic $v_2(L)$, and $H\cap
X=C_{x_1}+\dots+C_{x_d}$ as $H\cap Sec(v_2(\p^2))=\cup_{x\in
L}T_{v_2(x)}v_2(\p^2)$.
\end{example}

\begin{example}\label{ex:arrondo}
Identifying the planes (of a fixed family) contained in a smooth
quadric $Q\cong\g(1,3)$ in $\p^5$ with the set of lines passing
through a point of $\p^3$, we obtain the following two examples (see
\cite[cases 4) and 15) in p.~44]{a-s}):

(i) Let $C\subset\p^3$ be a twisted cubic, and let $X\subset\g(1,3)$ be
the family of secant lines to $C$. Then $X$ is the Veronese surface
$v_2(\p^2)\subset\p^5$.

(ii) Let $C\subset\p^3$ be an elliptic curve of degree $4$, and let
$X\subset\g(1,3)$ be the family of secant lines to $C$. Then $X$ is
a ruled surface with invariant $e=-1$ over $C$ and the embedding in
$\p^5$ is given by $|2C_0+\mathcal Lf|$, where $\deg(\mathcal L)=1$.
\end{example}

\section{Proof of the main results}\label{section:proof}
Let $X\subset\p^N$ be a smooth surface swept out by an algebraic
family $\{C_b\}_{b\in B}$ of plane curves. Since plane curves are
linearly normal, we can suppose without loss of generality that
$X\subset\p^N$ is linearly normal and non-degenerate. According to
\cite{seg2}, we also assume hereafter that $B$ is an integral curve.
We say that $\{C_b\}_{b\in B}$ \emph{does not form a fibration on
$X$} if there is no regular morphism $\pi:X\to B$ such that
$\pi^{-1}(b)=C_b$ for every $b\in B$. As $X$ is smooth and $C_b,
C_{b'}$ are algebraically equivalent for any two $b,b'\in B$, this
happens if and only if $C_b\cap C_{b'}\neq\emptyset$ for any two
$b,b'\in B$. Let $\varphi:B\to\g(2,N)$ be the map which associates
to each $b\in B$ the linear span $\langle C_b \rangle$ of $C_b$ in
$\p^N$, and let $\Sigma:=\varphi(B)$. The map $\varphi$ is constant
only if $X=\p^2$, so $\Sigma\subset\g(2,N)$ is an integral curve as
soon as $N\geq 3$. Let $V_{\Sigma}\subset\p^N$ denote the
$3$-dimensional subvariety swept out by the planes of $\Sigma$. We
divide the proof of Theorem \ref{thm:main} into two cases, as in
\cite{s-t}.

\subsection{$V_{\Sigma}$ is a cone}
We study separately two cases, according to the dimension of the
vertex:

\begin{proposition}\label{prop:L}
Let $X\subset\p^N$ be a smooth surface containing a $1$-dimensional
family of plane curves not forming a fibration, $N\geq 4$. If
$V_{\Sigma}\subset\p^N$ is a cone of vertex a line $L$, then
$V_{\Sigma}=S_{0,0,N-2}$ and $X$ is linked to $N-3$ rulings of
$V_{\Sigma}$ by the complete intersection of $V_{\Sigma}$ and a
hypersurface in $\p^N$.
\end{proposition}

\begin{proof}
Replacing $\g(2,4)$ by $\g(2,N)$ in the proof of \cite[Lemma
3.1]{s-t} we get that $\Sigma\subset\g(2,N)$ is a rational curve.
Therefore $V_{\Sigma}=S_{0,0,N-2}$, as $X\subset\p^N$ is
non-degenerate and linearly normal by assumption. Suppose first
$L\subset X\subset S_{0,0,N-2}$. It follows from \cite[Theorem
3.7]{ilic} that $\tilde X\in|\alpha H+F|$ and that $\pi_2|_{\tilde
X}:\tilde X\to X$ is an isomorphism. Since $H\sim(N-2)F$ we deduce
that $\tilde X\in|(\alpha+1) H - (N-3)F|$, i.e. $X$ is linked to
$N-3$ rulings of $S_{0,0,N-2}$ by the complete intersection of
$S_{0,0,N-2}$ and a hypersurface of degree $\alpha+1$ in $\p^N$.

Assume now $L\not\subset S$, and let $\{q_1,\dots,q_r\}:=X\cap L$.
Let us see that $X$ is singular at these points. Note that
$\pi_2|_{\tilde X}:\tilde X\to X$ contracts $\xi_i:=\pi^{-1}(q_i)$,
so we get $\sum_{i=1}^r\xi_i\sim r(H^2-b(H\cdot F))$. Let
$\p^1\times L:=\pi_2^{-1}(L)\subset\tilde S$, so $\p^1\times L\sim
H-bF$ (cf. \cite[Lemma 3.2]{ilic}). Moreover $(\p^1\times
L)\cdot\tilde X=\sum_{i=1}^r\xi_i$, and hence $(H-bF)\cdot(\alpha
H+\beta F)=r(H^2-b(H\cdot F))$. Therefore $\alpha
H^2-(b\alpha-\beta)(H\cdot F)=rH^2-rb(H\cdot F)$, whence $\alpha=r$
and $\beta=0$. So $X$ is the complete intersection of $S_{0,0,N-2}$
and a hypersurface of degree $\alpha$ in $\p^N$, and hence $X$ is
singular at $\{q_1,\dots,q_r\}$.
\end{proof}

On the other hand, if $V_{\Sigma}\subset\p^N$ is a cone of vertex a
point we get the following:

\begin{proposition}\label{prop:p}
Let $X\subset\p^N$ be a smooth surface containing a $1$-dimensional
family of plane curves not forming a fibration. If
$V_{\Sigma}\subset\p^N$ is a cone of vertex a point $p$, then
$V_{\Sigma}=S_{0,1,N-3}$ and $X$ is linked to $N-3$ rulings of
$V_{\Sigma}$ by the complete intersection of $V_{\Sigma}$ and a
hypersurface in $\p^N$.
\end{proposition}

\begin{proof}
Arguing as in Proposition \ref{prop:L}, we get that
$V_{\Sigma}=S_{0,a,b}\subset\p^{N=a+b+2}$. According to Theorem
\ref{thm:ilic} there are two possibilities. On the one hand, if
$\pi_2|_{\tilde X}:\tilde X\to X$ is an isomorphism then $\tilde
X\in|\alpha H+F|$. In this case, the pencil
$\{C_F\}_{F\in\p^1}:=\{\pi_2|_{\tilde X}(F|_{\tilde
X})\}_{F\in\p^1}$ of plane curves on $X$ forms a fibration, since
$(C_F)^2=(\pi^*_2(C_F))^2=(F|_{\tilde X})^2=0$. On the other hand,
if $\pi_2|_{\tilde X}:\tilde X\to X$ is the blowing-up of $X$ at $p$
then $a=1$ and $\tilde X\in|\alpha H-bF|$. Hence
$V_{\Sigma}=S_{0,1,b}$ and $X$ is linked to $b=N-3$ rulings of
$V_{\Sigma}$ by the complete intersection of $V_{\Sigma}$ and a
hypersurface of degree $\alpha$ in $\p^N$. In this case, the pencil
$\{C_F\}_{F\in\p^1}:=\{\pi_2|_{\tilde X}(F|_{\tilde
X})\}_{F\in\p^1}$ of plane curves on $X$ does not form a fibration
as we showed in Example \ref{ex:p}.
\end{proof}

\begin{remark}
Surfaces in Propositions \ref{prop:L} and \ref{prop:p} achieve the
greatest possible geometric genus $p_g$ (see \cite[p.~65
i)]{harris}, where $\varepsilon=0$ in both cases).
\end{remark}

\subsection{$V_{\Sigma}$ is not a cone}
In this case, we use Morin's result \cite{morin} on families of
intersecting planes quoted in Subsection \ref{subsect:morin}:

\begin{proposition}\label{prop:N=5}
Let $X\subset\p^N$ be a smooth surface containing a $1$-dimensional
family of plane curves not forming a fibration, $N\geq 5$. If
$V_{\Sigma}$ is not a cone then $N=5$ and one of the following
holds:

\begin{enumerate}
\item[(i)] $X\subset\p^5$ is the Veronese surface;
\item[(ii)] $X\subset\p^5$ is the second symmetric product of a smooth curve in
$\p^2$, and the embedding is given by the embedding
$\phi_2:S^2\p^2\hookrightarrow\p^5$;
\item[(iii)] $X\subset\p^5$ is a ruled surface with invariant $e=-1$ over an elliptic curve, and the embedding
is given by $|2C_0+\mathcal Lf|$, where $\deg(\mathcal L)=1$.
\end{enumerate}
\end{proposition}

\begin{proof}
Since the family $\{C_b\}_{b\in B}$ of plane curves does not form a
fibration on $X$, $N\geq 5$ and $V_{\Sigma}$ is not a cone, we
deduce that $\Sigma\subset\g(2,N)$ is a non-elementary family of
intersecting planes. Therefore, $N=5$ and $\Sigma\subset\g(2,5)$ is
contained in one of the families of Theorem \ref{thm:morin}.

If $\Sigma\subset\g(2,5)$ is as in Theorem \ref{thm:morin}(i), then
$X=v_2(\p^2)$. This is due to the fact that any two planes of
$\Sigma$ intersect in a point of $v_2(\p^2)$.

Assume now that $\Sigma\subset\g(2,5)$ is as in Theorem
\ref{thm:morin}(ii). Let $\varphi_2:\p^2\hookrightarrow\g(2,5)$
denote the embedding defined in Remark \ref{rem:families}, and let
$\varphi_2^{-1}(\Sigma)\subset\p^2$. Then
$X=\phi_2(\varphi_2^{-1}(\Sigma))\subset\p^5$, where
$\phi_2:\p^2\hookrightarrow\p^5$ is the embedding given in Example
\ref{ex:phi}. In particular, $\Sigma$ is necessarily smooth.

Finally, if $\Sigma\subset\g(2,5)$ is as in Theorem
\ref{thm:morin}(iii), we can suppose that every plane of $\Sigma$
corresponds to the set of lines passing through a point of $\p^3$
via $\varphi_3:\p^3\hookrightarrow\g(2,5)$ (cf. Remark
\ref{rem:families}). Therefore, $X\subset\g(1,3)$ is a smooth
congruence given by the family of secant lines to the (necessarily
smooth) curve $\varphi_3^{-1}(\Sigma)\subset\p^3$. Moreover,
$\varphi_3^{-1}(\Sigma)\subset\p^3$ cannot have trisecant lines as
otherwise $X\subset\g(1,3)$ would be singular. Then it is well known
that either $\varphi_3^{-1}(\Sigma)\subset\p^3$ is the twisted cubic
(and hence $X\subset\p^5$ is the Veronese surface), or else
$\varphi_3^{-1}(\Sigma)\subset\p^3$ is an elliptic quartic. In that
case, $X\subset\p^5$ is a ruled surface with invariant $e=-1$ over
an elliptic curve and the embedding is given by $|2C_0+\mathcal
Lf|$, where $\deg(\mathcal L)=1$, as we pointed out in Example
\ref{ex:arrondo}.
\end{proof}

%\begin{lemma}
%For any plane $\Lambda\subset\p^5$, let $\Omega_{\Lambda}=\{\Pi\mid
%\Lambda\cap\Pi\ne\emptyset\}\subset\g(2,5)$ denote the Schubert
%variety of planes meeting $\Lambda$. Then $\Pi\in\Omega_{\Lambda}$
%is a smooth point if and only if $\dim(\Pi\cap\Lambda)=0$, and
%$\Lambda$ is a singular point of $\Omega_{\Lambda}$ of multiplicity
%$3$.
%\end{lemma}

%\begin{proof}
%The first statement is well known. The second one follows from
%\cite[Theorem 1]{r-z}, where
%$d=3,n=6,\textbf{i}=(3,5,6),\textbf{j}=(1,2,3)$ and $s_1=s_2=s_3=0$.
%\end{proof}

Summing up:

\begin{proof}[Proof of Theorem \ref{thm:main}]
If $V_{\Sigma}\subset\p^N$ is a cone then we get (i) by Propositions
\ref{prop:L} and \ref{prop:p}. If $V_\Sigma\subset\p^N$ is not a
cone and $N=4$, then $X$ is either the rational normal scroll, or a
quintic elliptic scroll, or the projected Veronese surface by
\cite[Theorem 0.3]{b-s} (cf. \cite[Theorem 4.10 and Conjecture 4.13]{s-t}). The rational normal scroll and the
Veronese surface actually contain a $2$-dimensional family of plane
curves and they can be described as in case (i). More precisely,
$S_{1,2}\subset\p^4$ is linked to a plane by the complete
intersection of $S_{0,1,1}$ (resp. $S_{0,0,2}$) and a quadric of
$\p^4$, and $v_2(\p^2)\subset\p^5$ is linked to two planes by the
complete of $S_{0,1,2}$ and a quadric of $\p^5$. If
$V_\Sigma\subset\p^N$ is not a cone and $N\geq 5$, then we get the
remaining cases by Proposition \ref{prop:N=5}.
\end{proof}

\begin{remark}
Ruled surfaces with invariant $e=-1$ over elliptic curves also
appear in Theorem \ref{thm:main}(ii) when the plane curve is a
smooth cubic, and the embedding in $\p^5$ is given by $|3C_0|$ (cf.
Example \ref{ex:phi}).
\end{remark}

We now prove Corollary \ref{cor:main} by reducing the problem to the
case of surfaces:

\begin{proof}[Proof of Corollary \ref{cor:main}]
Let $Y\subset\p^N$ be a non-degenerate manifold of dimension $n$
swept out by a $1$-dimensional family $\{D_b\}_{b\in B}$ of
hypersurfaces, i.e. $\langle D_b\rangle=\p^n$, not forming a
fibration.

Let us show first that a surface as in cases (ii)-(iii) of Theorem
\ref{thm:main} cannot be a linear section of $Y$. To this aim we can
assume $n=3$. Let $X:=Y\cap H$ be a smooth non-degenerate surface
obtained by intersecting $Y$ with a general hyperplane $H$ in
$\p^N$. Then $C_b:=D_b\cap H$ is a plane curve for every $b\in B$
and the family $\{C_b\}_{b\in B}$ does not form a fibration on $X$.
To get a contradiction, assume $X$ as in cases (ii)-(iii) of Theorem
\ref{thm:main}. Then $D_b\cdot D_{b'}\cdot H=C_b\cdot C_{b'}=1$ for
any two $b,b'\in B$, whence $D_b\cap D_{b'}$ is a line in $\p^N$ and
$Y$ is swept out by a $2$-dimensional family $\mathcal
F\subset\g(1,N)$. Therefore $Y\subset\p^N$ is a cone over $X$, as otherwise $H$ contains some line
$L\in\mathcal F$ (and hence $L\subset X$, which is not possible). Let us prove this statement. Consider the dual family ${\mathcal
F}^*\subset\g(N-2,N)$, and let $V_{{\mathcal
F}^*}\subset{\p^N}^*$ denote the union of the $2$-dimensional family of $\p^{N-2}$'s. If $H$ does not contain any line
$L\in\mathcal F$ then $\dim V_{{\mathcal
F}^*}\leq N-1$, and hence $V_{{\mathcal
F}^*}=\p^{N-1}$. Consequently, there exists a point of $\p^N$ contained in every line $L\in\mathcal
F$ and $Y\subset\p^N$ is a cone over $X$.

So we deduce by intersecting $Y\subset\p^N$ with $n-2$ general
hyperplanes in $\p^N$ that $X\subset\p^{N-n+2}$ is contained in a
$3$-dimensional cone as in Theorem \ref{thm:main}(i). Hence
$Y\subset\p^N$ is contained in an $(n+1)$-dimensional cone
$V\subset\p^N$ with vertex $L$, and $\dim(L)=l$, where $n-2\leq
l\leq n-1$. We claim that $n=3$ and $l=1$. We argue as in
\cite[Lemmas 3.4 and 3.5]{sierra}. The embedded tangent space
$T_vV=\p^{n+1}\subset\p^N$ to $V$ at a general $v\in V$ is tangent
to $V$ along $\langle L,v\rangle/L$, whence $T_vV$ is tangent to $Y$
along $Y_v:=\langle L,v\rangle\cap Y$. Since $\dim(Y_v)=l$, Zak's
theorem on tangencies \cite[Ch.~I, Corollary 1.8]{zak2} yields
$l\leq n+1-n=1$. So $n=3$ and $l=1$. Therefore, the linearly
normal embedding of $Y$ is contained in a rational normal scroll
$S_{0,0,a,b}\subset\p^{a+b+3}$ and $D_b\cap D_{b'}=L \subset Y$.
Hence $\pi_2|_{\tilde Y}:\tilde Y\to Y$ cannot be an isomorphism, so
\cite[Claim 3.6]{ilic} yields $a=b=1$. In that case, $Y\subset\p^5$
is linked to a $\p^3$ by the complete intersection of the rank-$4$
quadric cone $S_{0,0,1,1}$ and a hypersurface in $\p^5$.
\end{proof}

In particular, Corollary \ref{cor:main} immediately yields the following
characterization:

\begin{corollary}
The Segre embedding $\p^1\times\p^2\subset\p^5$ is the only
non-degenerate manifold $X\subset\p^N$ of dimension $n\geq 3$ and
codimension $N-n\geq 2$ swept out by a $1$-dimensional family of
hypersurfaces $\{D_b\}_{b\in B}$ such that $D_b\cap D_{b'}$ moves on
$X$.
\end{corollary}

\section{Final remarks: a tribute to Ugo Morin}\label{section:tribute}
The purpose of this section is to point out some consequences of
Morin's results on families of intersecting linear spaces. In
\cite{morin2}, Morin extended his previous result on families of
intersecting planes in the following way:

\begin{theorem}[Morin \cite{morin2}]\label{thm:morinnn}
Let $\Sigma\subset\g(k,N)$ be a non-degenerate (i.e. not contained
in any $\g(k,N-1)$) integral subvariety of positive dimension such
that any two $k$-planes of $\Sigma$ intersect. Then (up to some
elementary exceptions) $N\leq k(k+3)/2$, with equality if and only
if $\Sigma$ is contained in one of the following families:
\begin{enumerate}
\item[(i)] the $\infty^2$ $k$-planes containing the rational normal curves of the $k$-Veronese
surface $v_k(\p^2)$ in $\p^{k(k+3)/2}$;
\item[(ii)] the $\infty^k$ tangent spaces to the Veronese manifold $v_2(\p^k)$ in
$\p^{k(k+3)/2}$;
\item[(iii)] the $\infty^{k+1}$ $k$-planes contained in $\g(1,k+1)\subset\p^{k(k+3)/2}$.
\end{enumerate}
\end{theorem}

In view of Theorem \ref{thm:morinnn}, we can definitely say that the
following result was already known by Morin:

\begin{proposition}\label{prop:coarseee}
Let $X\subset\p^N$ be a (maybe singular) non-degenerate surface
swept out by a $1$-dimensional family of curves, each of them
spanning a $k$-plane, such that any two of them intersect. Then
either $V_{\Sigma}\subset\p^N$ is a cone, or $N\leq k(k+3)/2$, with
equality if and only if one of the following holds:
\begin{enumerate}
\item[(i)] $X$ is the $k$-Veronese surface $v_k(\p^2)\subset\p^{k(k+3)/2}$;

\item[(ii)] $X$ is the second symmetric product of an irreducible curve in
$\p^k$, and the embedding in $\p^{k(k+3)/2}$ is given by
$\phi_k:S^2\p^k\hookrightarrow\p^{k(k+3)/2}$ (cf. Example
\ref{ex:phi});

\item[(iii)] $X$ is given by the secant lines to an irreducible curve of
$\p^{k+1}$, embedded by the Pl\"ucker embedding of $\g(1,k+1)$ in
$\p^{k(k+3)/2}$.
\end{enumerate}
\end{proposition}

\begin{proof}
If any two curves of the family intersect then we obtain a family of
intersecting $k$-planes, and we can apply Theorem \ref{thm:morinnn}.
The elementary cases in Theorem \ref{thm:morinnn} are the following
(see \cite[p.~186]{morin2}): either $k$-planes contained in a $\p^N$
with $N\leq 2k$, or $k$-planes containing, respectively, a
$(k-m)$-plane with $1\leq m\leq k$ such that any two $(k-m)$-planes
intersect. In the second case, either $k=m$ and hence
$V_{\Sigma}\subset\p^N$ is a cone, or else $X\subset\p^N$ is
degenerate. Otherwise, if the family is non-elementary, we get
$N\leq k(k+3)/2$, with equality if and only if
$\Sigma\subset\g(k,N)$ is as in cases (i)-(iii) of Theorem
\ref{thm:morinnn}. Now we can argue as in Proposition
\ref{prop:N=5}.
\end{proof}

\begin{remark}
We actually obtain smooth surfaces in Proposition
\ref{prop:coarseee} if and only if either $C\subset\p^k$ is smooth
in (ii), or $C\subset\p^{k+1}$ is smooth and has no trisecant lines
in (iii).
\end{remark}

\begin{remark}\label{rem:italian}
Concerning Proposition \ref{prop:coarseee}, Morin actually stated a
result in a more general setting at the end of \cite{morin} (cf.
\cite[Proposition 0.18]{c-c-ml}):

`` In conclusion: \emph{An irreducible algebraic surface containing an $\infty^1$ algebraic system of (irreducible) algebraic curves intersecting pairwise in just one variable point:}

a) \emph{such that for a general point of the surface there pass more than two of these curves, is a rational surface and the algebraic system of curves is contained in a homaloidal net;}

b) \emph{such that for a general point of the surface there pass exactly two of these curves, is the surface obtained as the pairs of points of a general curve of the system.}"

%In conclusione: \emph{Una superficie algebrica irriducibile che
%possiede un sistema algebrico $\infty^1$ di curve algebriche
%(irriducibili) a due a due incidenti in un solo punto variabile:}

%a) \emph{tali che per un punto generico della superficie passano
%pi\`u di esse, \`e una superficie razionale ed il sistema algebrico
%di curve \`e contenuto in una rete omaloidica;}

%b) \emph{tali che per un punto generico della superficie passano due
%di esse soltanto, \`e la superficie imagine delle coppie di punti di
%una curva generica del sistema dato.}
\end{remark}

\begin{remark}
A further application of Theorem \ref{thm:morin} was given by
Beauville in the study of rank-$3$ vector bundles and theta functions
on curves of genus $3$ (see \cite{beau}). A more recent application
has been given by O'Grady in \cite{o'grady}.
\end{remark}

We would like to finish this section by showing the connection
between families of intersecting linear spaces and the theory of secant
defective varieties:

\begin{remark}
Theorem \ref{thm:morinnn} appears to be related to the classification of Scorza varieties in the particular
case of secant defect $\delta=1$ (see \cite[Ch.VI, \S2]{zak2}),
where Severi's characterization of the Veronese surface was extended
to higher dimensions. More precisely, Zak proved that if $X\subset\p^N$ is a non-degenerate manifold of dimension $n$ that can be isomorphically projected into $\p^{2n}$ then $N\leq n(n+3)/2$, and equality holds if and only if $X=v_2(\p^n)$. According to Terracini's Lemma, if $X\subset\p^N$ can be isomorphically projected into $\p^{2n}$ then the family of embedded tangent spaces to $X\subset\p^N$ intersect pairwise. So one could deduce Zak's result from Theorem \ref{thm:morinnn} after excluding the elementary cases described in the proof of Proposition \ref{prop:coarseee}.
\end{remark}

\bibliography{bibfile}
\bibliographystyle{amsplain}
\end{document}